\documentclass[12pt]{amsart}
\usepackage{amssymb,amsthm,amsmath,amsfonts,euscript,delarray}
\usepackage{latexsym,amstext,amscd}
\usepackage{color}
\usepackage[all]{xy}

\newtheorem{theorem}{Theorem}

\newtheorem{remark}[theorem]{Remark}

\newcommand{\Z}{\mathbb{Z}}

\def\hull#1{\langle#1\rangle}

\def\nbd{neighbourhood}

\begin{document}


\baselineskip=17pt

\title[Countably infinite bounded abelian groups admit no non-discrete locally minimal group topologies]{Countably infinite bounded abelian groups admit no non-discrete locally minimal group topologies}

\author{Dekui Peng}

\date{December 29, 2019}

\begin{abstract}
In this note we show that if $G$ is a countably infinite abelian group such that $nG=0$ for some integer $n$, then the only locally minimal group topology on $G$ is the discrete one.
\end{abstract}
\maketitle

A Hausdorff topological group $(G, \tau)$ is called locally minimal if there exists a \nbd\ $U$ of the identity such that for every weaker Hausdorff group topology $\tau'$ on $G$, if $U$ is a $\tau'$-\nbd\ of the identity, then $\tau'=\tau$.

In \cite{ACDD}, Au{\ss}enhofer a et al. posed the following question (see also \cite{DM}):

Does the group $\bigoplus_{\omega}\Z(2)$ admit a non-discrete locally minimal group topology?

Here we will answer this question in negative.

\begin{theorem}Let $n$ be a positive integer and $G$ a countably infinite abelian group such that $nG=0$, then the only locally minimal group topology on $G$ is the discrete topology.\end{theorem}
\begin{proof}By way of contradiction we assume that $G$ carries a non-discrete locally minimal group topology.
Then every \nbd\ of 0 is infinite.
According to \cite[Lemma 2.3]{DHPXX}, there exists a closed \nbd\ $U$ of 0 such that every closed subgroup of $G$ contained in $U$ is minimal.
By Zorn's Lemma we know that there is a maximal element $K$ in the poset of subgroups of $G$ contained in the interior of $U$.
We will see that $K$ is infinite.

Indeed, for every $k\in K$, there exists a \nbd\ $V_k$ of $0$ with $k+V_k\subset U$.
Let $V=\bigcap_{k\in K}V_k$, then we obtain that $K+V\subset U$.
If $K$ were finite, then $V$ would be a \nbd\ of 0.
So we can choose a \nbd\ $W$ of 0 with $\underbrace{W+W+...+W}_{n~times}\subset V$.
Since $W$ is infinite, $W\setminus K$ is non-empty.
Take $x\in W\setminus K$, then $\hull{x}\subset V$ and hence, $K+\hull{x}\subset U$.
Evidently $K+\hull{x}$ properly contains $K$, this contradicts to the maximality of $K$.

By the above argument we see that the closure of $K$, $\overline{K}$ is a countably infinite closed subgroup of $G$ contained in $U$, so it is minimal.
According to the famous Prodanov-Stoyanov theorem, $\overline{K}$ is precompact.
Then the  completion $H$ of $\overline{K}$ is compact.
Since $nH=0$, there exists a prime divisor $p$ of $n$ such that $L:=\{x\in H:px=0\}$ is infinite.
Note that $L$ is also compact, so $|L|\geq 2^\omega$ (see \cite[Corollary 5.2.7.]{AT}).
By minimal criterion, $\overline{K}$ should be essential in $H$, i.e., every non-trivial closed subgroup of $H$ meets $\overline{K}$ non-trivially.
Therefore, $L=Soc(L)\leq \overline{K}$.
This contradicts to the countability of $\overline{K}$.
\end{proof}

\begin{remark}{\em Dikranjan noted that Protasov also found a negative solution to this question while his solution has not been published yet.}\end{remark}

{\bf Acknowledgements.} I am grateful to Professor Dikran Dikranjan for his guidance and suggestions.
Also I would like to thank Professor Vladimir Kadets who found a mistake in an earlier version.


\bigskip

\end{document}